\documentclass[11pt]{amsart}
\usepackage{amscd,amssymb}
\usepackage[colorlinks,linkcolor=blue,citecolor=blue,urlcolor=red]{hyperref}

  {\end{tabular}\par\medskip}
  
\newcommand{\uHom}{\operatorname{\underline{Hom}}}
\newcommand{\Ind}{\operatorname{Ind}}
\newcommand{\End}{\operatorname{End}}

\newcommand{\Coker}{\operatorname{Coker}}
\newcommand{\Ab}{\operatorname{Ab}}
\renewcommand{\Vec}{\operatorname{\mathbf{Vec}}}

\newcommand{\by}{\xrightarrow}

\newcommand{\iso}{\by{\sim}}

\newcommand{\sA}{\mathcal{A}}
\newcommand{\sB}{\mathcal{B}}
\newcommand{\sE}{\mathcal{E}}
\newcommand{\Q}{\mathbf{Q}}
\newcommand{\fS}{\mathfrak{S}}
\renewcommand{\phi}{\varphi}
\newcommand{\un}{\mathbf{1}}

\newcounter{spec}
\newenvironment{thlist}{\begin{list}{\rm{(\roman{spec})}}%
{\usecounter{spec}\labelwidth=20pt\itemindent=0pt\labelsep=10pt}}%
{\end{list}}

\newtheorem{lemma}{Lemma}[section]
\newtheorem{prop}[lemma]{Proposition}
\newtheorem{thm}[lemma]{Theorem}
\newtheorem{cor}[lemma]{Corollary}
\theoremstyle{definition}
\newtheorem{defn}[lemma]{Definition}
\theoremstyle{remark}
\newtheorem{rk}[lemma]{Remark}
\newtheorem{rks}[lemma]{Remarks}

\newtheorem{exs}[lemma]{Examples}

\begin{document}

\title{Exactness and faithfulness of monoidal functors}
\author{Bruno Kahn}
\address{IMJ-PRG\\ Case 247\\4 place Jussieu\\75252 Paris Cedex 05\\France}
\email{bruno.kahn@imj-prg.fr}
\date{December 2, 2021}
\begin{abstract}Inspired by recent work of Peter O'Sullivan, we give a condition under which a faithful monoidal functor between abelian $\otimes$-categories is exact.\end{abstract}
\maketitle

\section*{Introduction} A well-known result of Deligne (\cite[Prop. 1.19]{dm}, \cite[Cor. 2.10]{dfest}, \cite[remark in 0.9]{dtens}) says that an exact $\otimes$-functor from a rigid abelian symmetric monoidal category in which $\End(\un)$ is a field to a (nonzero) abelian symmetric monoidal category is faithful.  Coulembier recently proved that the converse is true, by topos-theoretic methods (\cite[Cor. 4.4.4]{coul1}, see also \cite[Th. 2.4.1]{coul2}).  In \cite[Lemma 10.7]{os}, O'Sull\-iv\-an gives a different proof of this converse in the special case of super Tannakian categories. These converses are remarkable, as they give a drastically simple way of filling the gap in Saavedra's proof that Tannakian categories are classified by gerbes (see \cite[end of Question (3.15)]{dm}).

The aim of this note is to generalise Deligne's and O'Sullivan's results by weakening their hypotheses, in particular, removing a rigidity hypothesis in the second case. The first case (Proposition \ref{p1}) is very simple and is a footnote to Deligne's argument; the second, Theorem \ref{p2} d), is more elaborate.  Parts of this work apply to the quasi-abelian situation of \cite{andreslope}, see Example \ref{ex1} and Remark \ref{r2}.

The condition for a monoidal functor $F$ to preserve monomorphisms or epimorphisms turns out to be all-pervasive here. The novel fact is the existence of a notion of super-dimension in a very general context (Definition \ref{d4} and Lemma \ref{l4}): it plays a key rôle in the proof of Theorem \ref{p2}.

\section{On Corollary 2.10 of \cite{dfest}}

\begin{lemma}\label{l5} Let $\sA,\sB$ be two monoidal categories, with $\sA$ rigid, and let $F:\sA\to \sB$ be a strong monoidal functor. Then $F$ is faithful $\iff$ $F$ separates morphisms with domain $\un$ $\iff$ $F$ separates morphisms with codomain $\un$.
\end{lemma}

\begin{proof} Let $f,g:M\rightrightarrows N$ be two parallel morphisms in $\sA$, and $\tilde f,\tilde g:\un\rightrightarrows M^\vee\otimes N$ be the morphisms corresponding to $f$ and $g$ by adjunction, where $M^\vee$ is the dual of $M$. Suppose that $F(f)=F(g)$. Then $F(\tilde f) = \widetilde{F(f)}=\widetilde{F(g)}=F(\tilde g)$, hence $\tilde f=\tilde g$ and $f=g$. Dual reasoning for ``codomain''.
\end{proof}

\begin{prop}\label{p1} In Lemma \ref{l5}, let $\sA,\sB$ and $F$ be additive (the tensor products are biadditive). Suppose that
\begin{thlist}
\item any nonzero morphism of $\sA$ with domain $\un$ is a monomorphism;
\item $F$ preserves monomorphisms.
\end{thlist}
Then $F$ is faithful.\\
Same conclusion when replacing ``monomorphism'' by ``epimorphism''.
\end{prop}

Indeed, the condition of Lemma \ref{l5} is verified.\qed

\begin{cor}\label{c2} If $\sA,\sB$ are abelian, $\sB\ne 0$, $\End_\sA(\un)$ is a field and $F$ preserves monomorphisms (equivalently, epimorphisms), it is faithful.
\end{cor}

Indeed, $\un$ is simple \cite[Prop. 1.17]{dm}, hence (i) holds in Proposition \ref{p1}.\qed

(See Lemma \ref{c1} d) and Proposition \ref{l2} c), e) for partial converses.)

\section{On Lemma 10.7 of \cite{os}}


All categories are $\Q$-linear, pseudo-abelian symmetric monoidal categories, and all functors are $\Q$-linear $\otimes$-functors (= strong symmetric monoidal functors). We refer to \cite{dtens} for Schur funcctors.

\subsection{Property $S$}

\begin{defn}\label{d1} We say that a category $\sA$ has \emph{property $S$} if the following condition holds:
\begin{quote}
A morphism $f:M\to N$ is mono if and only if ($S(f) = 0$ $\Rightarrow$ $S(M) = 0$) for any Schur functor $S$.
\end{quote}
\end{defn}

\begin{rk}\label{r1} $\sA^{op}$ has property $S$ if and only if the following condition holds:
\begin{quote}
A morphism $f:M\to N$ is epi if and only if ($S(f) = 0$ $\Rightarrow$ $S(N) = 0$) for any Schur functor $S$.
\end{quote}
We call this \emph{property $S^{op}$}.
\end{rk}

\begin{lemma}\label{c1} Let $F:\sA\to \sB$ be faithful. \\
a) If $\sB$ verifies the ``if'' part of property $S$, so does $\sA$.\\
b) If $F$ preserves monomorphisms and $\sB$  verifies the ``only if'' part of property $S$, so does $\sA$.\\
c) If $F$ preserves monomorphisms and $\sB$  has property $S$, so does $\sA$.\\
d) If $\sA$  verifies the ``only if'' part of property $S$ and $\sB$ verifies the ``if'' part of property $S$, then $\sA$ has property $S$ and $F$ preserves monomorphisms. \\
e) If $\sA$  verifies the ``if'' part of property $S$ and $\sB$ verifies the ``only if'' part of property $S$, then $F$ reflects monomorphisms.
\end{lemma}

\begin{proof} This is a simple matter of logic, but we give details to avoid any confusion.

a) Let $f:M\to N$. Suppose that $F(f)$ is not mono. By hypothesis, there is a Schur functor  $S$ such that $S(F(f))=0$ but $S(F(M))\ne 0$. Then $S(M)\ne 0$ but $S(f)=0$ by the faithfulness of $F$. If now $f$ is not mono, then $F(f)$ is not mono again by the faithfulness of $F$, so the above applies. 

b) Let $f:M\to N$ which is mono, so that  $F(f)$ is mono. Let $S$ be a Schur functor such that $S(f)=0$. Then  $S(F(f))=0$, hence $S(F(M))=F(S(M))=0$ by hypothesis, and $S(M)=0$ by the faithfulness of $F$.

c) follows immediately from a) and b).

d) The first claim follows from a). For the second one, let $f:M\to N$ which is mono. If $F(f)$ is not mono, then by the proof of a) there is a Schur functor  $S$ such that $S(f)=0$ but $S(M)\ne 0$. This contradicts the ``only if'' part of property $S$.

e) Suppose that $F(f)$ is mono. Let $S$ be such that $S(f)=0$. Then $S(F(f))=F(S(f))=0$, hence $F(S(M))=S(F(M))=0$ by the ``only if'' part of property $S$. Thus $S(M)=0$ by the faithfulness of $F$. This shows that $f$ is mono by the ``if'' part of property $S$.
\end{proof}

\begin{prop}\label{p0} For any $\sA$,\\
a) Kernels and cokernels are stable under tensor product by dualisable objects.\\
b) Monomorphisms with dualisable domain and epimorphisms with dualisable codomain are stable under tensor products, tensor powers and the application of Schur functors.\\
c) If $\sA$ is rigid, it verifies the ``only if'' part of properties $S$ and $S^{op}$ and its $\otimes$-structure is exact.
\end{prop}

\begin{proof} a) More generally, cokernels are stable under tensor product with an object $P$ such that the internal Hom $\uHom(P,Q)$ is representable for any $Q$, because this tensor product is then right exact as a left adjoint. Similarly for kernels, assuming that internal coHom's (Moh objects in Saavedra's terminology) are representable.


b) follows from a) by using the symmetry of the tensor product and the fact that monomorphisms and epimorphisms are stable under composition (the same generalisation as in a) applies); c) follows trivially from a) and b) (closed and coclosed would be sufficient).
\end{proof}


\begin{prop}\label{p3} Suppose $\sA$ abelian semi-simple, Schur-finite (any object is killed by some Schur functor) and such that $M\otimes N\ne 0$ if $M\ne 0$ and $N\ne 0$ (see Remark \ref{r3} below). Then $\sA$ has properties $S$ and $S^{op}$.\qed
\end{prop}

\begin{proof} We only deal with property $S$. Let $f:M\to N$ be a morphism in $\sA$.

\emph{Only if}: by semi-simplicity, $f$ mono has a retraction $g$, thus $S(f)$ has the retraction $S(g)$ for all $S$.

\emph{If}: still by semi-simplicity, write $f$ as a direct sum of an isomorphism $f':M'\iso N'$ and a $0$ morphism $f'':M''\to N''$. By \cite[Cor. 1.7]{dtens}, the set of partitions $\lambda$ such that $S_\lambda(M')=0$ is a sieve for the opposite order to the inclusion of Young diagrams $[\lambda]$; since $M'$ is Schur-finite, we may therefore find $\lambda$ such that $S_\lambda(M')=0$ but $S_\mu(M')\ne 0$ for any $\mu$ such that $[\mu]\subsetneq [\lambda]$. By \cite[(1.5.1) et (1.8.1)]{dtens}, $S_\lambda(M)$ contains $S_\mu(M')\otimes M''$ as a direct summand for such $\mu$ with $|\mu| = |\lambda|-1$. If $f$ is not mono, i.e. if $M''\ne 0$, this summand is nonzero by hypothesis, so $S_\lambda(M)\ne 0$. 
Let $n=|\lambda|$. Then $f^{\otimes n}$ is the direct sum of the isomorphism ${f'}^{\otimes n}$ and a $0$ morphism; applying the idempotent defining $S_\lambda$ to this decomposition, we get $S_\lambda(f) = 0$.
\end{proof}

\begin{cor}\label{l1} The $\otimes$-category $\Vec_K^\pm$ of finite-dimensional super-vector spaces over a field $K$ of characteristic $0$ has properties $S$ and $S^{op}$.\qed
\end{cor}

The following lemma is obvious:

\begin{lemma} \label{d3} Let $M\in \sA$. Write $\Sigma(M)$ for the (possibly empty) set of Schur functors $S$ such that  $S(M)=0$. Then $\Sigma(M)\subseteq \Sigma(F(M))$ for any $\otimes$-functor $F:\sA\to \sB$, with equality if $F$ is faithful.\qed
\end{lemma}

\subsection{Pre-tannakian categories}

\begin{defn}\label{d2} Let $\sA$ be a $\Q$-linear symmetric monoidal category.\\ 
a) $\sA$ is \emph{proto-tannakian}  if it admits a faithful $\otimes$-functor $\omega:\sA\to \Vec_K^\pm$, where $K$ is a field of characteristic $0$. We call $\omega$ a \emph{fibre functor}.\\
b) $\sA$ is \emph{left} (resp. \emph{right}) \emph{pre-tannakian} if we can choose $\omega$ preserving monomorphisms (resp.  epimorphisms).\\ 
c) $\sA$ is \emph{pre-tannakian} if it is left and right pre-tannakian.
\end{defn}

\begin{exs}\label{ex1} a) By \cite[Lemma 2.1.5 (2)]{andreslope}, a quasi-tannakian category in the sense of loc. cit., Def. 2.1.1 is pre-tannakian.\\
b) If $\sA$ is abelian, rigid and Schur-finite (super Tannakian in the sense of \cite[Def. 10.2]{os}), with $k=\End_\sA(\un)$ a field, it is pre-tannakian. Indeed, we may assume $k$ algebraically closed; by \cite[Prop. 2.1]{dtens}, $\sA$ admits an exact $\otimes$-functor $\omega$ to the category of super-modules over a super commutative $k$-algebra $R$, whose values are (by rigidity)  projective. Thus we may replace $R$ with a field $K$ by specialisation. We then get the faithfulness of $\omega$ from Corollary \ref{c2}.\\
c) By \cite[Th. 10.10]{os}, b) extends to the case where $\sA$ is only integral in the sense of loc. cit., beg. \S3, i.e. $f\otimes g=0$ $\Rightarrow$ $f=0$ or $g=0$ (pseudo-Tannakian in the terminology of loc. cit.). (Conversely, any proto-tannakian category is clearly integral and Schur-finite.)
\end{exs}

\begin{rk}[cf. proofs of Lemma \ref{l5} and Cor. \ref{c2}] \label{r3} If $\sA$ is abelian and rigid and $\End_\sA(\un)$ is a field, it is integral. Indeed, let $f:A\to A'$ and $g:B\to B'$ be two nonzero morphisms. Their adjoints $\tilde f:\un\to A^\vee\otimes A'$, $\tilde g:\un\to B^\vee\otimes B'$ are nonzero, hence monomorphisms since $\un$ is simple \cite[Prop. 1.17]{dm}. By the exactness of the tensor product (Proposition \ref{p0} c)), $\tilde f\otimes \tilde g$ is a monomorphism, hence nonzero. But it is the adjoint of $f\otimes g$. In particular, the last hypothesis of Propositon \ref{p3} is automatic if $\sA$ is rigid (or contained in a rigid abelian category).

On the other hand, O'Sullivan proves in \cite[Lemma 3.1]{os} a very strong result: $\sA$ is fractionally closed in the sense of loc. cit., Sect. 3, bot. p. 6.
\end{rk}

\begin{prop}\label{l2} a) Any left (resp. right) pre-tannakian category has property $S$ (resp. $S^{op}$).\\
b) If $\sA$ is left pre-tannakian and $\sB$ verifies the ``if'' (resp. the ``only if'') part of property $S$, any faithful functor $F:\sA\to \sB$ preserves (resp. reflects) monomorphisms. Same statement when replacing left by right, monomorphism by epimorphism and $S$ by $S^{op}$.\\
c) If $\sA$ and $\sB$ are pre-tannakian, any faithful functor $F:\sA\to \sB$ (in particular, any fibre functor $\omega:\sA\to \Vec_K^\pm$) preserves and reflects monomorphisms and epimorphisms.\\ 
d) If $\sA$ is pre-tannakian, its tensor product preserves monomorphisms and epimorphisms.\\
e) Any rigid proto-tannakian category is pre-tannakian.
\end{prop}

\begin{proof} a) follows from Lemma \ref{c1} c) and Corollary \ref{l1}. In b), ``preserves'' follows from a) and Lemma \ref{c1} d), while ``reflects'' follows from a)  and Lemma \ref{c1} e). c) follows from b) and Corollary \ref{l1}. d) follows from c) and the exactness of the tensor product in $\Vec^\pm_K$; e) follows from Proposition \ref{p0} c),  Corollary \ref{l1} and Lemma \ref{c1} d).
\end{proof}


\subsection{Super-dimensions}

\begin{defn}\label{d4} Let $\sA$ be proto-tannakian, and let $M\in \sA$. We define the \emph{super-dimension} $\dim^\pm(M)$ of $M$ as the super-dimension of $\omega(M)$ for $\omega$ as in Definition \ref{d2}. 
\end{defn}

\begin{lemma} \label{l4} $\dim^\pm(M)$ does not depend on the choice of $\omega$. It is preserved by any faithful functor (between proto-tannakian categories). Moreover:
\begin{enumerate}
\item $\dim^\pm(M)$ is additive in $M$; $M=0$ $\iff$  $\dim^\pm(M) =0$.
\item $M$ invertible $\Rightarrow$ $\dim^\pm(M) = 1|0$ or $0|1$; the converse is true if $M$ is dualisable.
\item If $\sA$ is left (resp. right) pre-tannakian, then $\dim^\pm(M) \le \dim^\pm(N)$ (resp. $\dim^\pm(M)\ge \dim^\pm(N)$) for any monomorphism (resp. epimorphism) $f:M\to N$.
\item If there is equality in (3), then $f$ is also epi (resp. mono).
\item If $\dim^\pm(M)=m|n$ and $M$ is dualisable, then $tr(1_M)=m-n$.
\end{enumerate}
\end{lemma}

\begin{proof} By \cite[Cor. 1.9]{dtens}, $\dim^\pm(M)$ is the pair $m|n$ such that $S_\lambda(\omega(M))\allowbreak = 0$ for a partition $\lambda$ if and only if $[\lambda]$ contains the rectangular diagram with $n + 1$ rows and $m + 1$ columns; the two claims then follow from Lemma \ref{d3}. Next, (1), (3) and (5) hold in a category of the form $\Vec_K^\pm
$, hence in general; (4) follows in the same way by using Proposition \ref{l2} c). Finally, in (2) the implication $\Rightarrow$ is trivial; for the converse, the unit $\un\to M^\vee\otimes M$ and the counit $M\otimes M^\vee\to \un$ become inverse isomorphisms after applying $\omega$, hence are inverse isomorphisms by faithfulness.
\end{proof}

\begin{rk}\label{r2} By Example \ref{ex1} a), Proposition \ref{l2} c) and Lemma \ref{l4} (3) give an elementary proof of \cite[Lemma 2.1.5 (2)]{andreslope}. Similarly, Proposition \ref{p0} a) gives an elementary proof of part (1) of this lemma.
\end{rk}

\emph{In the sequel, we assume that $\sA$ is pre-tannakian.}

\begin{lemma}\label{l3} Let $\sE:0\to M'\by{i} M\by{p} M''\to 0$ be a $0$ sequence in $\sA$, with $i$ a mono and $p$ an epi. Then we have
\begin{equation}\label{eq1}
\dim^\pm(M) \ge \dim^\pm(M')+\dim^\pm(M'')
\end{equation}
with equality if and only if $\omega(\sE)$ is exact for some ($\iff$  any) fibre functor $\omega:\sA\to \Vec_K^\pm$.
\end{lemma}

\begin{proof} This follows from Proposition \ref{l2} c).
\end{proof}

\begin{defn}\label{d5} A sequence $\sE$ as in Lemma \ref{l3} is \emph{dim-exact} if \eqref{eq1} is an equality.
\end{defn}

\begin{prop}\label{p4}  Assume that $\sA$ is abelian and that its tensor product is right exact (e.g. that $\sA$ has internal Homs). Then $\dim^\pm(M)\le \dim^\pm(M')+\dim^\pm(M'')$ for any exact sequence $\sE:M'\to M\to M''\to 0$.
\end{prop}

\begin{proof} (cf. \cite[proof of Prop. 1.19]{dtens}.) Let first $\sE_i:M'_i\to M_i\to M''_i\to 0$ be a family of exact sequences, for $i\in\{1,\dots,n\}$. Using the right exactness of the tensor product, a standard diagram chase yields an exact sequence
\[M'_1\otimes M_2\oplus M_1\otimes M'_2\to M_1\otimes M_2\to M''_1\otimes M''_2\to 0\]
hence, by induction, an exact sequence
\begin{multline*}
M'_1\otimes M_2\otimes\dots \otimes M_n \oplus M_1\otimes M'_2\otimes \dots \otimes M_n\oplus\dots\oplus  M_1\otimes M_2\otimes \dots \otimes M'_n\\
\to M_1\otimes\dots \otimes  M_n\to M''_1\otimes\dots \otimes  M''_n\to 0.
\end{multline*}
Taking $\sE_1=\dots =\sE_n= \sE$, we thus get an exact sequence
\begin{multline}\label{eq3}
M'\otimes M^{\otimes n-1} \oplus M\otimes M'\otimes  M^{\otimes n-2}\oplus\dots\oplus  M^{\otimes n-1} \otimes M'\\
\to M^{\otimes n}\to M''^{\otimes n}\to 0.
\end{multline}

For the action of $\fS_n$ on \eqref{eq3} by permutations, the first term is isomorphic to $\Ind_{\fS_{n-1}}^{\fS_n}(M'\otimes M^{\otimes n-1})$. Iterating, we see that $M^{\otimes n}$ admits a $\fS_n$-equivariant filtration of length $n+1$ whose $p$-th quotient  is isomorphic to $\Ind_{\fS_p\times \fS_{n-p}}^{\fS_n}( {M'}^{\otimes p}\otimes {M''}^{\otimes n-p})$, except for $p=n$ where it is only a quotient of ${M'}^{\otimes n}$, cf. \cite[bot. p. 232]{dtens}. As in loc. cit., Prop. 1.8, it follows that if $\lambda$ is a partition of $n$, $S_\lambda(M)$ has a filtration of length $n+1$ whose $p$-th quotient is isomorphic to
\[[M',M'']_p=\bigoplus_{\mu,\nu} (S_\mu(M')\otimes S_\nu(M''))^{[\lambda:\mu,\nu]}\]
where $|\mu|=p$, $|\nu|=n-p$, except for $p=n$ when it is only a quotient of $S_\lambda(M')$. (Here, $[\lambda:\mu,\nu]$ is the multiplicity described in \cite[1.5]{dtens}.)

Now we have, by \cite[Prop. 1.8]{dtens}:
\[\omega\left(\bigoplus_{p=0}^n [M',M'']_p\right) \simeq S_\lambda(\omega(M'\oplus M'')) \simeq \omega(S_\lambda(M'\oplus M''))\]
hence
\begin{multline*}
S_\lambda(M)\ne 0 \Rightarrow \bigoplus_{p=0}^n [M',M'']_p\ne 0 \Rightarrow \omega\left(\bigoplus_{p=0}^n [M',M'']_p\right)\ne 0\\
\iff \omega(S_\lambda(M'\oplus M''))\ne 0 \Rightarrow S_\lambda(M'\oplus M'')\ne 0
\end{multline*}
whence $\dim^\pm(M)\le \dim^\pm(M'\oplus M'')=\dim^\pm(M')+\dim^\pm(M'')$ (Lemma \ref{l4} (1)), as desired.
\end{proof}

\subsection{Main theorem} Recall that $\sA$ is pre-tannakian. Let $\sE$ be as in Lemma \ref{l3}.

\begin{thm}\label{p2} a)  For any faithful functor $F:\sA\to \sB$ (with $\sB$ pre-tannakian) $\sE$ dim-exact (Definition \ref{d5}) $\iff$ $F(\sE)$  dim-exact.\\
b) Suppose that $i$ has a cokernel. If $\sE$ is dim-exact, so is the sequence $0\to M'\by{i} M\to \Coker i\to 0$, and the induced morphism $\Coker i\to M''$ is mono-epi. Same statement, mutatis mutandis, assuming that $p$ has a kernel.\\
c) If $\sA$ is abelian,  dim-exact $\Rightarrow$ exact, and the converse is true if its tensor structure is right exact.\\
d) Let $\sA,\sB$ be abelian pre-tannakian categories. If the tensor structure of $\sA$ is right exact, any faithful $\otimes$-functor $F:\sA\to \sB$ is exact.
\end{thm}

\begin{proof} a) follows from Proposition \ref{l2} a) and b), and Lemma \ref{l4}.

b) The morphism $\Coker i\to M''$ is clearly epi, hence $\dim^\pm \Coker i \ge \dim^\pm M''$ by Lemma \ref{l4} (3); if $\sE$ is dim-exact, we must have equality, which shows the first claim. The second one follows from Lemma \ref{l4} (4). Same argument assuming that $p$ has a kernel.


In c), $\Rightarrow$ follows from b) and $\Leftarrow$ follows from Proposition \ref{p4}. 


d) now follows from a) and c).
\end{proof}
\bigskip

\begin{rks} a) The proof of Theorem \ref{p2} d) is inspired from the key idea in the proof of \cite[Lemma 10.7]{os}. We recover this lemma when $\sA$ and $\sB$ are rigid with $\End(\un)$ a field, see Proposition \ref{p0} c) and Example \ref{ex1} b).\\
b) I don't know if the right exactness hypothesis is necessary in Theorem \ref{p2} c) and d) (see Proposition \ref{l2} d)); it holds for the abelian envelope $\Ab(\sA)$ of an additive $\otimes$-category $\sA$ from \cite[Prop. 1.8 and 1.10]{bhp} (but $\Ab(\sA)$ may not be pre-tannakian even if $\sA$ is.)\\
c) I don't know if Theorem \ref{p2} c) and d) extend to the quasi-tannakian categories of \cite{andreslope} (see Example \ref{ex1}).
\end{rks}

\enlargethispage*{20pt}

\end{document}